\documentclass[12pt]{amsart}

\usepackage[utf8]{inputenc}
\usepackage{amsfonts}
\usepackage{amsthm}
\usepackage{amssymb}
\usepackage{amsmath}
\usepackage{amscd}
\usepackage{latexsym,dsfont}
\usepackage{bbm}
\usepackage{mathrsfs}
\usepackage{esint} %seyma added%

\usepackage{times}
\usepackage{microtype}
\usepackage{setspace}
\usepackage[margin=1.2in]{geometry}

\usepackage{cite}

\usepackage[colorlinks=true, pdfstartview=FitV, linkcolor=blue,
citecolor=blue, urlcolor=blue]{hyperref}

\newtheorem{theorem}{Theorem}[section]
\newtheorem{lemma}[theorem]{Lemma}
\newtheorem{prop}[theorem]{Proposition}
\newtheorem{cor}[theorem]{Corollary}

\newtheorem{defn}[theorem]{Definition}

\newtheorem{hyp}[theorem]{Hypothesis}

\theoremstyle{definition}

\theoremstyle{remark}
\newtheorem{remark}[theorem]{Remark}

\numberwithin{equation}{section}
 \allowdisplaybreaks

\def\Xint#1{\mathchoice
   {\XXint\displaystyle\textstyle{#1}}%
   {\XXint\textstyle\scriptstyle{#1}}%
   {\XXint\scriptstyle\scriptscriptstyle{#1}}%
   {\XXint\scriptscriptstyle\scriptscriptstyle{#1}}%
   \!\int}
\def\XXint#1#2#3{{\setbox0=\hbox{$#1{#2#3}{\int}$}
     \vcenter{\hbox{$#2#3$}}\kern-.5\wd0}}

\def\avgint{\Xint-}

\DeclareMathOperator{\Div}{div}
\DeclareMathOperator{\supp}{supp}
\DeclareMathOperator{\dist}{dist}

\DeclareMathOperator{\grad}{\nabla}

\newcommand{\op}{{\mathrm{op}}}

\DeclareMathOperator{\lip}{\mathrm{Lip}}
\newcommand{\pp}{{p(\cdot)}}

\newcommand{\N}{\mathbb N}
\newcommand{\C}{\mathbb{C}}
\newcommand{\Z}{\mathbb Z}
\newcommand{\R}{\mathbb{R}}

\newcommand{\Q}{\mathcal{Q}}

\newcommand{\A}{\mathcal A}
\newcommand{\W}{\mathcal W}
\newcommand{\Hh}{\mathcal H}

\newcommand{\vecf}{\mathbf f}
\newcommand{\vecg}{\mathbf g}

\newcommand{\G}{\mathcal{G}}

\title
{On $H=W$ in Banach function spaces}

\author[\c{C}etin, Cruz-Uribe, Rodney] {  \c{S}eyma \c{C}etin, David Cruz-Uribe OFS,  Scott Rodney}

\address{\c{S}eyma \c{C}etİn \\
Dept. of Basic Sciences, Istanbul Gelisim University, Avcilar, Türkiye
}
\email{seycetin@gelisim.edu.tr}

\address{David Cruz-Uribe, OFS \\
Dept. of Mathematics \\
University of Alabama \\
 Tuscaloosa, AL 35487, USA}

\email{dcruzuribe@ua.edu}

\address{Scott Rodney\\
Dept. of Mathematics, Physics and Geology \\ 
Cape Breton University \\
Sydney, NS B1M1A2, Canada} 

\email{scott\_rodney@cbu.ca}

\thanks{The first author is supported by the TUBITAK 2219 - International Postdoctoral Research Fellowship Program for Turkish Citizens. The second author is partially supported by a Simons Foundation Travel Support for Mathematicians Grant and by NSF Grant DMS-2349550.  The third author is partially supported by an NSERC Development Grant and an NSERC Discovery Grant, Canada. The authors would like to thank Zoe Nieraeth for several helpful discussions about Banach function spaces and for pointing out several results in the literature. The authors also thank Lars Diening for pointing out several references to clarify the history of some of the conditions in Section~2. }

\keywords{Sobolev spaces, convolution operators, averaging operators, Banach function spaces}
\subjclass{46E35, 46E30, 42B25, 42B35}

\makeatletter
\def\l@subsection{\@tocline{2}{0pt}{4pc}{5pc}{}}
\makeatother

\begin{document}

\begin{abstract}
    In this paper we prove ``$H=W$" in the context of a Banach function space $X(\Omega)$.  Let $\Omega$ be a subset of $\R^n$ and denote by $W^1_X(\Omega)$ the collection of all those $f\in X(\Omega)$ whose distributional derivatives $\partial_jf$ are contained in $X(\Omega)$.  Our main result provides a small collection of “universal” hypotheses on $X(\Omega)$ that ensure $W^1_X(\Omega)$ is equal to $H^1_X(\Omega)$, the formal closure of $\textrm{Lip}(\Omega)\cap W^1_X(\Omega)$ with respect to the norm
    \[\|f\|_{W^1_X(\Omega)} = \|f\|_{X(\Omega)} + \|\nabla f\|_{X(\Omega)}.\]
    The main theorem has two corollaries.  The first gives a slightly stronger set of hypotheses for ``$H=W$", and the second gives the density of $C^\infty_c(\R^n)$ in $W^1_X(\R^n)$.
\end{abstract}

\maketitle

\section{Introduction}

In this paper we extend the classical ``$H=W$" result of Meyers and Serrin~\cite{MR164252}  to the setting of Sobolev spaces defined over Banach function spaces.  Given an open subset $\Omega\subset \R^n$, $m\in \N$, and $1\leq p<\infty$, they proved that $H^{m,p}(\Omega)=W^{m,p}(\Omega)$,
where $W^{m,p}(\Omega)$ is the Banach space of  measurable functions $u$ on $\Omega$ that have distributional derivatives up to order $m$ in $L^p(\Omega)$, and $H^{m,p}(\Omega)$ is closure in $W^{m,p}(\Omega)$ of $C^\infty \cap W^{m,p}(\Omega)$ with respect to the norm
$$\|u\|_{W^{m,p}(\Omega)} = \sum_{|\alpha|\leq m} \|D^\alpha u\|_{L^p(\Omega)}. $$ 

To put our results in context, we give a brief history of interest in this problem, which was motivated through the study of PDEs.   Fabes, Kenig, and Serpioni, in their celebrated paper~\cite{MR643158}, developed a regularity theory for weak solutions of linear second order degenerate elliptic equations.  Given a bounded domain $\Omega\subset \R^n$ and $w\in A_2$, they studied the divergence form operator $Lu = -\Div(Q\nabla u)$, where $Q$ satisfies the  weighted ellipticity condition 
\begin{equation*}
\lambda w(x)|\xi|^2 \leq \langle Q(x)\xi,\xi\rangle  \leq \Lambda w(x)|\xi|^2
\end{equation*}
for all $\xi\in\R^n$ and a.e. $x\in \Omega$.  They defined weak solutions to be elements of the weighted space $H^1(\Omega, w)$,  the closure of $C^\infty(\overline{\Omega})$ with respect to the weighted Sobolev norm 
$$\|u\|_{H^1(\Omega,w)} = \left(\int_\Omega |u|^2\,wdx\right)^\frac{1}{2}
+ \left(\int_\Omega |\nabla u|^2\,wdx\right)^\frac{1}{2}.$$   
Given the identity ``$H=W$", it is natural to  ask if $H^1(\Omega,w)=W^{1,2}(\Omega,w)$  where $W^{1,2}(\Omega,w)$ is the collection measurable functions with distributional derivatives in $L^2(\Omega,w)$. (Somewhat surprisingly, this question was not addressed in~\cite{MR643158}.)   Kipel\"ainen~\cite[Theorem 2.5]{MR1246890} showed the following.  For $1\leq p<\infty$ and $w\in A_p$, define the weighted Sobolev space $W^{1,p}(\Omega,w)$ to be the collection of functions $u \in L^p(\Omega,w)$ whose distributional gradient $\nabla u$ is in $L^p(\Omega,w)$, and let $H^{1,p}(\Omega,w)$ be the completion of $C^\infty(\R^n)\cap W^{1,p}(\Omega,w)$ with respect to the norm
\[ \|u\|_{W^{1,p}(\Omega,w)} = \left(\int_\Omega |u|^p\,wdx\right)^\frac{1}{p}
+ \left(\int_\Omega |\nabla u|^p\,wdx\right)^\frac{1}{p}.\]
Then $W^{1,p}(\Omega,w)=H^{1,p}(\Omega,w)$.

Working in a very different direction,  Zhikov~\cite[Theorem 4.1]{MR1669639} proved $H=W$ in   weighted Sobolev spaces without assuming an $A_p$ condition.  Let $\Omega\subset \R^n$ be a bounded domain, and  let $F\subset \Omega$ be closed.  Let $\rho$ be  a non-negative,  locally integrable weight $\Omega$ whose degeneracies are contained in $F$: that is, for $\epsilon>0$ sufficiently small,  $1<c_1(\epsilon)\leq \rho(x)\leq c_2(\epsilon)<\infty$ for $x\in \Omega\setminus F_\epsilon$ where $F_\epsilon = \{x\in\Omega~:~\dist(x,F)\leq \epsilon\}$. With the same definition as given by Kipel\"ainen, he showed that $W^{1,2}(\Omega,w)=H^{1,2}(\Omega,w)$ if $F$ has capacity zero and $\rho$ satisfies the auxiliary estimate 
$$\rho(x)\leq \frac{C}{\textrm{Cap}(F_\epsilon)}$$
for $x\in \Omega\setminus F_\epsilon$ for all $\epsilon>0$ sufficiently small.

The results of  Kipel\"ainen were generalized to collections of Lipschitz vector fields by Franchi, Serapioni, and Serra Cassano~\cite{franchi1997} and by Garofalo and Nhieu~\cite{Garofalo1998}. Let $\mathcal{X}=\{X_1,...,X_m\}$ be a collection of vector fields $X_j = (c_{j1},...,c_{jn})\cdot \nabla$, where the coefficient functions $c_{ij}$ are Lipschitz in $\Omega$.  Given a weight $w$,  define $W^{1,p}_\mathcal{X}(\Omega,w)$ to be the collection of measurable functions $u\in L^p(\Omega,w)$ with $X_ju \in L^p(\Omega,w)$, $j=1,\ldots,m$.  Define $H^{1,p}_\mathcal{X}(\Omega,w)$ to be the closure of $C^\infty(\Omega)\cap W^{1,p}_\mathcal{X}(\Omega,w)$ with respect to the norm
$$\|u\|_{W^{1,p}_\mathcal{X}(\Omega,w)} = \|u\|_{L^p(\Omega,w)} + \sum_{i=1}^m \|X_j u\|_{L^p(\Omega,w)}.$$
Given this, they proved that $H^{1,p}_\mathcal{X}(\Omega,w) =W^{1,p}_\mathcal{X}(\Omega,w)$ whenever $w\in A_p$.

These results have been generalized in several directions.  We can rewrite $\|X_j u\|_{L^p(\Omega,w)}$ as $\|w^{1/p}X_j u\|_{L^p(\Omega)}$; written like this, the weight $w$ can be thought of as part of a rough vector field $\tilde{X}_j = w^{1/p}X_j$.  Sawyer and Wheeden~\cite{MR2574880} considered the problem of rough vector fields in the following setting.  Let  $(\Omega,\rho, |\cdot|)$ be a space of homogeneous type with doubling constant $D$. Define $\mathcal{X}=\{X_1,...,X_m\}$ to be a collection of vector fields $X_j = (v_{j1},...,v_{jn})\cdot \nabla$ such that for $1\leq j\leq m$ and $1\leq i\leq n$,  $v_{ji}\in L^D(\Omega)$ and $\nabla v_{ji}\in L^D(\Omega)$.  Further suppose that if $1\leq p<D$, then the Sobolev inequality 
$$\|\varphi\|_{L^q(E)} \leq C\left(\|\varphi\|_{L^p(\Omega)}+ \| \mathcal{X}\varphi\|_{L^p(\Omega)}\right)$$
holds for all $\varphi\in Lip_0(E)$, where $E\subset \Omega$ is open and $\frac{1}{q}=\frac{1}{p}-\frac{1}{D}$.  The constant $C=C(p,E,\mathcal{X})$ is independent of $\varphi$.
With these assumptions, they proved that $H^{1,p}_\mathcal{X}(\Omega) = W^{1,p}_\mathcal{X}(\Omega)$.  They also gave a counter example to show that it is necessary to assume that the Sobolev inequality holds.  See~\cite[section 5.1]{MR2574880} for further details. 

Moen and the second and third authors~\cite{MR3544941}, rather than considering vector fields, instead used matrix weights in the matrix $\A_p$ class. Given an $n\times n$, symmetric, positive semi-definite matrix $W$,  let $w=|W|_{\op}$ be the largest eigenvalue of $W$.  For $1\leq p<\infty$, they defined $W^{1,p}(W,\Omega)$ to be the space of all functions $f$ with distributional gradients that satisfy
\[ \|f\|_{\W^{1,p}(\Omega,W)} = \|f\|_{L^p(\Omega,w)} + \|W^{1/p}\grad f\|_{L^p(\Omega)} < \infty, \]
and they defined $\Hh^{1,p}(\Omega.W)$ to be the closure of $C^\infty(\Omega)\cap \W^{1,p}(\Omega,W)$ in $\W^{1,p}(\Omega,W)$.  They showed that if $W\in \A_p$, then $\W^{1,p}(\Omega,W) = \Hh^{1,p}(\Omega,W)$.  As an application they studied the partial regularity of degenerate $p$-Laplacians and mappings of finite distortion.

The second author and Fiorenza~\cite{MR3026953}, and separately, Diening, {\em et al.}~\cite{MR2790542} studied $H=W$ in the context of the variable exponent Lebesgue spaces $L^\pp(\Omega)$. These are Banach function spaces that generalize the classical $L^p$ spaces, replacing the constant exponent $p$ with a variable exponent function $\pp$.  They showed that  under appropriate regularity conditions on the exponent function $\pp$, the original result of Meyers and Serrin extended to this setting.  More precisely, they showed that the Sobolev space  $W^{m,p(\cdot)}(\Omega)$ consisting of $L^{p(\cdot)}(\Omega)$ functions with distributional derivatives in $L^{p(\cdot)}(\Omega)$ is equal to $H^{m,p(\cdot)}(\Omega)$,  the closure of $C^\infty(\Omega)\cap W^{m,p(\cdot)}(\Omega)$ with respect to the
norm
\[ \|f\|_{W^{m,\pp}(\Omega)} = \sum_{|\alpha|\leq m} \|D^\alpha f\|_{L^\pp(\Omega)}. \]
More recently, the second author and Penrod~\cite{MR4777231} extended the results of~\cite{MR3544941} to matrix weighted spaces defined over the variable Lebesgue spaces using matrix $\A_\pp$ weights, a generalization of the matrix $\A_p$ weights to this setting.  As an (implicit) corollary to their work, they extended the results of Kipel\"ainen to this setting, proving that if the exponent function $\pp$ is log-H\"older continuous, and if $w\in \A_\pp$ (the scalar generalization of $A_p$ weights introduced in~\cite{MR2927495}), then $W^{1,\pp}(\Omega,w)= H^{1,\pp}(\Omega,w)$.  (We refer the reader to the above references for definitions and further information.)

The proximate motivation for our work was a series of papers by the first author, Bilalov, Mamedov, and others~\cite{MR4431572, MR4778460, MR4853049, MR4443241, MR4449429, MR4801675}.  They were interested in extending the classical theory of elliptic equations to the more general setting where the data on the right-hand side of the equation is drawn from a Banach function space $X(\Omega)$. As part of their hypotheses, they needed to restrict to functions contained in a subspace that was, essentially, the closure of $C^\infty(\Omega)\cap X(\Omega)$ in the corresponding Sobolev space $W_X^1(\Omega)$.  These papers built upon previous work where $X(\Omega)$ was taken to be a specific kind Banach function space:  e.g., the grand-Lebesgue spaces~\cite{MR4611542}, 
Morrey spaces~\cite{MR4486271}, and rearrangement invariant spaces~\cite{MR4778460}.  Given this work, a natural question is to determine broad conditions such that  ``$H=W$" in a Banach function space.

The goal of this paper is to answer this question and generalize the results discussed above by finding a small collection of ``universal" hypotheses on a Banach function space $X(\Omega)$ so that the first order Sobolev space, $W_X^1(\Omega)$,  is equal to $H^1_X(\Omega)$, the closure of smooth functions in $W_X^1(\Omega)$.  For clarity, we summarize the required hypotheses here; we explain them in detail below in Definitions \ref{defn:fatousr}, \ref{defn:abscont}, and \ref{defn:strong-muckenhoupt}.

\begin{hyp} \label{hyp:main}
    Given a domain $\Omega$ and a Banach function space $X(\Omega)$, we assume that it has the following properties:
    \begin{itemize}
        \item The Fatou property holds;
        \item The norm on $X(\Omega)$ is absolutely continuous;
        \item The strong Muckenhoupt condition holds.
    \end{itemize}
\end{hyp}
\begin{theorem} \label{thm:H=W}
     Given a domain $\Omega$ and a Banach function space $X(\Omega)$ that satisfies Hypothesis~\ref{hyp:main}, $H^1_X(\Omega)=W^1_X(\Omega)$. 
\end{theorem}

As an immediate corollary, we have that we can replace the strong Muckenhoupt condition with two related conditions; see Definitions \ref{defn:muckenhouptsr} and \ref{PropertyG}.

\begin{cor} \label{cor:H=W-propG}
    Given a domain $\Omega$, let $X(\Omega)$ be a Banach function space that satisfies Hypothesis~\ref{hyp:main} with the strong Muckenhoupt condition replaced by the Muckenhoupt condition and property $\G$. Then $H^1_X(\Omega)=W^1_X(\Omega)$. 
\end{cor}

As another corollary, we show that if $\Omega=\R^n$, then smooth functions of compact support are dense.

\begin{cor} \label{cor:rn-compact-support}
Suppose the Banach function space $X(\R^n)$ satisfies Hypothesis~\ref{hyp:main}.  Then $C_c^\infty(\R^n)$ is dense in $W^1_X(\R^n)$.
\end{cor}

\begin{remark}
    In Theorem \ref{thm:H=W} we make no assumptions about the boundary of $\Omega$; we note that the same is true in the original work of Meyers and Serrin discussed above.
\end{remark} 

\begin{remark}
    In Theorem~\ref{thm:H=W} and Corollaries~\ref{cor:H=W-propG} and~\ref{cor:rn-compact-support} we have only considered the first order Sobolev space $W^1_X(\Omega)$. We have done so for clarity and conciseness.  However, for any $m\geq 1$ we could define the spaces $W^m_X(\Omega)$, and our arguments could then be readily modified to prove that $W^m_X(\Omega)= H^m_X(\Omega)$.  We leave the details to the interested reader.
\end{remark}

\begin{remark}
    In Theorem~\ref{thm:H=W} we assume that $X(\Omega)$ is a Banach function space.  It is an open question whether it is true in the more general setting of quasi-Banach function spaces (\textit{QBFS}).  As we will show below, a number of the results we need in the proof  hold in this setting.  However, in a key step (see the proof of Theorem~\ref{thm.conv.op.bdd.}) we need to use the triangle inequality (rather than the quasi-triangle inequality of a \textit{QBFS}) which only holds for Banach function spaces. It is not clear to us how to avoid this dependence.
\end{remark}

\begin{remark}
    It is an interesting open question whether the matrix weighted results in~\cite{MR3544941,MR4777231} could be extended to the setting of Banach function spaces, using the ``directional Banach function spaces" introduced by Nieraeth~\cite{MR4966577}.  Similarly, one can ask if the results in~\cite{franchi1997,Garofalo1998,MR2574880} can be extended to more general Banach function spaces.
\end{remark}

\medskip

 The assumption that $X(\Omega)$ has the Fatou property is standard; it is part of the definition of a Banach function space in~\cite{MR928802}.  The assumption that $X(\Omega)$ has absolutely continuous norm means that it is, in some sense, ``far" from $L^\infty$.  This assumption holds, for instance, in the weighted Lebesgue spaces, $1\leq p <\infty$, and the (weighted) variable Lebesgue spaces if the exponent function is uniformly bounded. 

\medskip

The strong Muckenhoupt condition is more restrictive.  A sufficient condition for it to hold is that the Hardy-Littlewood maximal operator is bounded on $X(\Omega)$ or on its K\"othe dual, $X'(\Omega)$.  This is the case, for instance, on the weighted Lebesgue spaces $L^p(w)$ if $1\leq p<\infty$ and $w\in A_p$; it is also true in the variable Lebesgue spaces $L^\pp(\Omega)$, provided that $\pp$ or its dual exponent $p'(\cdot)$ are uniformly bounded.  In fact, in both examples, if the spaces are reflexive, then the strong Muckenhoupt condition is equivalent to the boundedness of the maximal operator on $X(\Omega)$ and $X'(\Omega)$.  For further examples of spaces where the strong Muckenhoupt condition holds, and for a careful analysis of its relation to other conditions, we refer the reader to Nieraeth~\cite{MR4963357}.
%, where $p'(\cdot)$ is defined pointwise by $\frac{1}{p(x)}+\frac{1}{p'(x)}=1$.
\medskip

The remainder of this paper is organized as follows. In Section~\ref{section:prelim}, we give a precise definition of quasi-Banach function spaces (briefly, \textit{QBFS}) and the hypotheses we impose upon them.  We follow the definition of \textit{QBFS} given by Lorist and Nieraeth~\cite{MR4726599}, which includes the definition of Banach function spaces (\textit{BFS}) due to Bennett and Sharpley~\cite{MR928802}.  We take this approach to better illustrate where our hypotheses are used.  We note that
with our additional hypotheses, when $\Omega=\R^n$ our spaces are essentially the same as the ball quasi-Banach function spaces introduced by Sawano, {\em et al.}~\cite{MR3687096}.   This section also includes some tools like H\"older's inequality and basic results related to characteristic functions and averaging operators.  In section~\ref{section:sobolev} we define the Sobolev spaces  $W^1_X(\Omega)$ and $H^1_X(\Omega)$, and prove that $W^1_X(\Omega)$ is a Banach space, which  gives the inclusion $H^1_X(\Omega)\subset W^1_X(\Omega)$.   In Section~\ref{section:density}, we give several results about dense subsets in a \textit{BFS}.  Some of these are already found in the literature in various forms, but we gather them here with consistent hypotheses.  In Section~\ref{section:convolution} we prove several results about convolutions and approximate identities in a \textit{BFS}.  These are central to our proof of Theorem~\ref{thm:H=W} but are also of interest in their own right.   Finally, in Section~\ref{section:main-proof}, we prove Theorem~\ref{thm:H=W} by proving that $W^1_X(\Omega) \subset H^1_X(\Omega)$, and we prove Corollary~\ref{cor:rn-compact-support}.

Throughout this paper, we will use the following notation. The symbol $n$ is the dimension of the Euclidean space $\R^n$ and we denote the usual Euclidean norm by $|\cdot|$. We will take the domain of our functions to be an open,  connected set $\Omega \subseteq \R^n$. The set $\Omega$ need not, {\em a priori}, be bounded. The letter $C$ will denote a constant that may vary in value from line to line depending on underlying parameters. If we want to specify a particular dependence, we will write, for instance, $C(n)$. We will also use the convention that $1/{\infty} = 0$.

\section{Preliminaries}
\label{section:prelim}

In this section we define Banach function spaces and give some of their basic properties.  Let $L^0(\Omega)$ denote the collection of measurable functions (with respect to Lebesgue measure) on $\Omega$.

\begin{defn}\label{defn:QBFS}
    Given a domain $\Omega$, a quasi-Banach function space on $\Omega$, denoted $X(\Omega)$, is a vector subspace of $L^0(\Omega)$  that is equipped with a norm:  that is, a function $\|\cdot\|_X : X \rightarrow [0,\infty)$ such that 
    \begin{enumerate}
        \item $\|f\|_{X(\Omega)} \geq 0$ for all $f\in X(\Omega)$ and $\|f\|_{X(\Omega)}=0$ if and only if $f=0$ a.e.;
        \item $\|\alpha f\|_{X(\Omega)} =|\alpha|\|f\|_X$ for all $f\in X(\Omega)$ and $\alpha \in \R$;
        \item $\|f+g\|_{X(\Omega)}\leq M\big(\|f\|_{X(\Omega)}+\|g\|_{X(\Omega)}\big)$ for some $M\geq 1$ and for all $f,\,g\in X(\Omega)$. The infimum of all such constants is denoted by $M_X$.  If $M_X=1$, then $X(\Omega)$ is a Banach function space.
    \end{enumerate}
    We further assume that $X(\Omega)$ is complete with respect to convergence in this quasi-norm, and that it satisfies the following two properties:
    \begin{enumerate}
    \setcounter{enumi}{3}
        \item (ideal property)
         if $f\in X(\Omega)$ and $g\in L^0(\Omega)$ with $|g|\leq |f|$, then $g\in X(\Omega)$ and $\|g\|_{X(\Omega)}\leq \|f\|_{X(\Omega)}$;

        \item (saturation property) for every measurable $E\subseteq \Omega$ of positive measure, there exists a measurable $F\subseteq E$ of positive measure with $\chi_F \in X(\Omega)$. 
    \end{enumerate}
\end{defn}

\begin{remark}
    This definition of a quasi-Banach function space is taken from~\cite{MR4726599}.  In the case of a Banach function space, conditions $(1)-(4)$ agree with the definition in~\cite{MR928802}.  The saturation property is weaker than what is assumed in~\cite{MR928802} and we refer the reader to~\cite{MR4726599} for further details.   
\end{remark}

We now define the associate space of a \textit{QBFS}, also referred to as its K\"othe dual.

\begin{defn}
    Given a domain $\Omega$ and a quasi-Banach function space  $X(\Omega)$, the associate space $X'(\Omega)$ consists of all functions $g\in L^0(\Omega)$ such that $fg\in L^1(\Omega)$ for all $f\in X(\Omega)$.   We define
    \[ \|g\|_{X'(\Omega)} = \sup\bigg\{ \int_\Omega |fg|\,dx : f\in X(\Omega), \|f\|_{X(\Omega)}=1 \bigg\}.  \]
\end{defn}

\begin{remark} \label{remark:X'-BFS}
If $X(\Omega)$ is a \textit{BFS}, then so is $X'(\Omega)$:  see~\cite[Theorem~3.2]{MR4726599} and~\cite[Theorem~2.2]{MR928802}. However, this is not true for a general  \textit{QBFS}. It is straightforward to show that $\|\cdot\|_{X'(\Omega)}$ is a norm, and that with it $X'(\Omega)$ is complete and has the ideal property.  However, it may not have the saturation property:  see~\cite{MR4726599} for details.  
\end{remark}

We have a generalization of H\"older's inequality to a QBFS and its associate space.

\begin{lemma}\label{Hölder} 
    Let $X(\Omega)$ be a quasi-Banach function space. If $f\in X(\Omega)$ and $g\in X'(\Omega)$, then $fg$ is integrable and
    \begin{equation*}
        \int_{\Omega}|fg|\,dx\leq \|f\|_{X(\Omega)}\,\|g\|_{X'(\Omega)}.
    \end{equation*}
\end{lemma}

% \begin{remark}
%     As an immediate consequence of the triangle inequality $(3)$ above, we have that given any finite collection $\{f_n\}_{n=1}^N$ in $X(\Omega)$,
%     %
%     \[ \bigg\| \sum_{n=1}^N f_n \bigg\|_{X(\Omega)} \leq \sum_{n=1}^N M_X^n \|f_n\|_{X(\Omega)}. \]
%     %
% \end{remark}

We now define the additional hypotheses we require for Theorem~\ref{thm:H=W}.  There is some redundancy in the statement of Hypothesis~\ref{hyp:main}, which we will explain.  We chose to state it this way for clarity.

\begin{defn}\label{defn:fatousr}
    A quasi-Banach function space $X(\Omega)$ has the Fatou property if, given a sequence $\{f_n\}_{n=1}^\infty$ such that $\sup_n \|f_n\|_{X(\Omega)}<\infty$ and such that the $f_n$ increase pointwise to $f$, then $f\in X(\Omega)$  and 
    \[ \|f_n\|_{X(\Omega)}=\lim_{n\to \infty} \|f_n\|_{X(\Omega)}. \]
\end{defn}

\begin{remark}
    If a vector space equipped with a norm has the Fatou property, then it is complete:  see~\cite[Theorem~1.6]{MR928802} and~\cite[Remark~2.1]{MR4726599}.
\end{remark}

As a consequence of the Fatou property, we have the Lorentz-Luxemburg theorem.  If $X(\Omega)$ is a \textit{BFS} with the Fatou property, then $X'(\Omega)$ is also a \textit{BFS}, and its associate space, $X''(\Omega)$, satisfies $X''(\Omega)= X(\Omega)$ with equality of norms.  (See~\cite[Theorem~71.1]{MR222234} and~\cite[Theorem~2.7]{MR928802}.)  As an immediate consequence of this we get the following estimate.

\begin{lemma} \label{lemma:duality.ineq}
    Let $X(\Omega)$ be a Banach function space with the Fatou property.  Then for any $f \in X(\Omega)$, there exists $g\in X'(\Omega)$ with $\|g\|_{X'(\Omega)}\leq 1$ such that
    \begin{equation*}
        \|f\|_{X(\Omega)}\leq 2 \int_\Omega f(x) g(x) \,dx.
    \end{equation*}
\end{lemma}

\begin{defn}\label{defn:abscont}
    A quasi-Banach function space $X(\Omega)$ has an absolutely continuous norm if given any $f\in X(\Omega)$ and any collection of sets $\{E_n\}_{n=1}^\infty$ such that $E_{n+1}\subset E_n$ and $\bigcap_n E_n = \emptyset$, then 
    \[ \lim_{n\to\infty} \|f\chi_{E_n}\|_{X(\Omega)}=0. \]
\end{defn}

The existence of an absolutely continuous norm is equivalent to the norm being order continuous, which in turn is equivalent to the dominated convergence theorem in $X(\Omega)$.  For a proof of the next result, see~\cite[Prop.~3.6]{MR928802} and~\cite[Prop.~3.12]{MR4726599}.  

\begin{lemma} \label{lemma:dominated-conv}
   Let $X(\Omega)$ be a quasi-Banach function space with absolutely continuous norm.  Given any sequence $\{f_n\}_{n=1}^\infty$ in $X(\Omega)$ such that $f_n\to f$ pointwise a.e., and there exists $g\in X(\Omega)$ such that $|f_n|\leq g$, then 
   \[ \lim_{n\to \infty} \|f_n-f\|_{X(\Omega)} = 0. \]
\end{lemma}

\begin{defn}\label{defn:muckenhouptsr}
     A quasi-Banach function space $X(\Omega)$ satisfies the Muckenhoupt condition on $\Omega$, denoted by $X(\Omega)\in A(\Omega)$, if given any cube $Q\subset \Omega$, $\chi_Q \in X(Q)$ and $\chi_Q \in X'(Q)$, and 
     \[ [X(\Omega)]_{A(\Omega)} =\sup_{Q}|Q|^{-1}\|\chi_{Q}\|_{X(\Omega)}  \|\chi_{Q}\|_{X'(\Omega)} < \infty. \]
\end{defn}

\begin{remark}
    In the case $\Omega=\R^n$, the assumption that $\chi_Q$ is contained in $X(\Omega)$ and $X'(\Omega)$  means that these spaces are ball quasi-Banach function spaces as defined in~\cite{MR3687096}.
\end{remark}

The Muckenhoupt condition is closely related to the Muckenhoupt $A_p$ condition that plays a central role in the study of weighted norm inequalities.  In the general setting of Banach spaces, it was introduced by Berezhnoi~\cite{MR1622773} and plays an important role in the boundedness of many classical operators in a \textit{BFS}.  See~Nieraeth~\cite{MR4963357} for further details and references.  For us, the key role of the Muckenhoupt condition is that it characterizes the boundedness of averaging operators on $X(\Omega)$.   Given a cube $Q\subset \Omega$, define the averaging operator $A_Q$ by 
\begin{equation*}
    A_Qf(x)=\fint_Q f(y)dy\,\chi_Q(x).
\end{equation*}
The following result was proved by Nieraeth~\cite[Prop.~3.1]{MR4963357} for $\Omega=\R^n$, but the same proof holds for general $\Omega$.

\begin{theorem}\label{PropertyofAveragingOp}
     Let $X(\Omega)$ be a quasi-Banach function space.  Given a cube $Q\subset \Omega$,  $A_Q:X(\Omega)\rightarrow X(\Omega)$ if and only if $\chi_Q\in X(\Omega)$, $\chi_Q\in X'(\Omega)$.  Moreover,
    \begin{equation*}
        \|A_Q\|_{X(\Omega)} = |Q|^{-1}\|\chi_{Q}\|_{X(\Omega)}  \|\chi_{Q}\|_{X'(\Omega)}
    \end{equation*}
where $\|A_Q\|_{X(\Omega)}$ denotes the operator norm of $A_Q$.  Consequently,  $X(\Omega)$ satisfies the Muckenhoupt condition if and only if averaging operators are uniformly bounded on $X(\Omega)$ and
\[ \sup_Q  \|A_Q\|_{X(\Omega)} = [X(\Omega)]_{A(\Omega)}.\]
\end{theorem}

\begin{remark} \label{remark:ball-averages}
    It follows from Theorem~\ref{PropertyofAveragingOp} and the ideal property that averaging operators $A_B$, defined with respect to balls contained in $\Omega$, are uniformly bounded if we restrict to balls that are contained in cubes $Q\subset \Omega$.   Given a point $x\in \Omega$, let $\delta= \dist\{x, \partial\Omega\}$.  Then for any $r<\delta/\sqrt{n}$, $B=B(x,r)$ is contained in the cube $Q=Q(x,r)$ centered at $x$ with sidelength $2r$ and $Q(x,r)\subset \Omega$.  For such balls, we have that $|A_Bf|\leq c(n)A_Q(|f|)$, and so $\|A_B\|_{X(\Omega)} \leq c(n)[X(\Omega)]_{A(\Omega)}$.
    \end{remark}

    One consequence of the Muckenhoupt condition, or more properly, the assumption that $\chi_Q\in X(\Omega)$ for all $Q\in \Omega$, is the following lemma.  A version of this result for $\Omega=\R^n$ was proved in~\cite[Prop.~2.2]{MR4963357}.

  \begin{lemma}\label{charac.func.of.comp.set.bdd.}
    Given a quasi-Banach function space $X(\Omega)$, suppose $\chi_Q\in X(\Omega)$ for every cube $Q\subset \Omega$. If $K\subset \Omega$ is compact, then $\chi_K\in X(\Omega)$.
\end{lemma}

\begin{proof}
    Since $K$ is compact and $\Omega$ is open,  if we let $\delta =\dist(K,\partial \Omega)$, then $0<\delta<\infty$. For $x\in K$, let $Q_x=Q(x,r)$ be the open cube centered at $x$ with $r<\delta/\sqrt{n}$.  Then $Q_x\subset\Omega$. Further, since  $\mathcal{Q}=\{Q_x:x\in K\}$ is an open cover of $K$, there exists a finite subcover $\{Q_{x_j}\}_{j=1}^N$, with centres $x_1,...,x_N \in K$.  Since these cubes overlap we have that
    \[ \chi_K(x)\leq \sum_{j=1}^N\chi_{Q_{x_j}}(x) \]
    for every $x\in \Omega$.  Since each of $\chi_{Q_{x_j}}\in X(\Omega)$ so is their sum and, using the ideal property (see Definition \ref{defn:QBFS}) we see that $\chi_K \in X(\Omega)$.
\end{proof}  

\begin{remark} \label{remark:X'-BFS-2}
    As a consequence of Lemma~\ref{charac.func.of.comp.set.bdd.} we have that if $X(\Omega)$ is a quasi-normed vector space with the ideal property and which satisfies the Muckenhoupt condition, then $X(\Omega)$ satisfies the saturation property as well.  This follows from the fact that every set $F$ of positive measure contains a compact set of positive measure.  
    The property that $\chi_K\in X(\Omega)$ whenever $K$ is compact is still  weaker than the assumption that $\chi_E\in X(\Omega)$ for any set $E$ of finite measure used in~\cite{MR928802}. 
    \end{remark}
    
    \begin{remark}
        Note that if we assume that $X'(\Omega)$ is such that $\chi_Q\in X'(\Omega)$ for every cube $Q\subset \Omega$, then the proof of Lemma~\ref{charac.func.of.comp.set.bdd.} shows that $\chi_K\in X'(\Omega)$, and so $X'(\Omega)$ satisfies the saturation property.  Therefore, by Remark~\ref{remark:X'-BFS}, $X'(\Omega)$ is a Banach function space.
\end{remark}

Given a collection $\Q$ of pairwise disjoint cubes $Q\subset \Omega$, define the averaging operator $A_\Q$ by
\begin{equation*}
    A_\Q f(x)=\sum_{Q\in \Q} \fint_{Q} f(y)dy\,\chi_Q(x).
\end{equation*}

\begin{defn} \label{defn:strong-muckenhoupt}
    A quasi-Banach function space $X(\Omega)$ satisfies the strong Muckenhoupt condition on $\Omega$, denoted by $X(\Omega)\in A^s(\Omega)$, if given any cube $Q\subset \Omega$, $\chi_Q \in X(Q)$ and $\chi_Q \in X'(Q)$, and 
    \[ [X(\Omega)]_{A^s(\Omega)} 
    = \sup_{\Q} \|A_{\Q}\| < \infty, \]
    where the supremum is taken over all collections of disjoint cubes $\Q$ contained in $\Omega$ and $\|A_{\Q}\|$ is the operator norm of $A_\Q$. 
\end{defn}

The strong Muckenhoupt condition was introduced by Diening~\cite[Definition~3.1]{MR2166733} in the context of Musielak-Orlicz spaces, a large family of BFS.  It is immediate that the strong Muckenhoupt property implies the Muckenhoupt property.  The converse is false:  see Diening, {\em et al.}~\cite[Theorem~5.3.4]{MR2790542}.  The following property was introduced by Berezhnoi~\cite{MR1622773} to close the gap between the two conditions.

\begin{defn}\label{PropertyG}
    A quasi-Banach function space  $X(\Omega)$ has  property $\G(\Omega)$  if there exists a constant $G$ such that for all $f\in X(\Omega)$, $g\in X'(\Omega)$, and all  collections $\Q$ of pairwise disjoint cubes $Q\subseteq\Omega$,
    \begin{equation*}
        \sum_{Q\in \Q}\big\|\chi_{Q}\,f\big\|_{X(\Omega)}\,\big\|\chi_{Q}\,g\big\|_{X'(\Omega)}\leq G\|f\|_{X(\Omega)}\,\|g\|_{X'(\Omega)}.
    \end{equation*}
    We denote the infimum of all such constants $G$ by $[X(\Omega)]_{\G(\Omega)}$.
\end{defn}

The following result shows that the Muckenhoupt condition and property $\G$ imply the strong Muckenhoupt condition; Corollary~\ref{cor:H=W-propG} then follows immediately from Theorem~\ref{thm:H=W}. It is due originally to Berezhnoi; a stronger result is proved for $\Omega=\R^n$ in~\cite[Theorem~3.9]{MR4963357}; for the convenience of the reader we include the short proof. (Also see~\cite[Section~7.3]{MR2790542}.)

\begin{prop}
    Let $X(\Omega)$ be a Banach function space that satisfies the Fatou property, the Muckenhoupt condition, and  property $\G(\Omega)$. Then for every collection $\Q$ of pairwise disjoint cubes in $\Omega$,
    \begin{equation}
        \|A_\Q f\|_{X(\Omega)}\leq 2[X(\Omega)]_{A(\Omega)}[X(\Omega)]_{\G(\Omega)} \|f\|_{X(\Omega)}. \label{av.op.bdd}
    \end{equation}

\end{prop}
\begin{proof}
    Fix a collection $\Q$ of pairwise disjoint cubes contained in $\Omega$, and fix $f\in X(\Omega)$.  By Lemma~\ref{lemma:duality.ineq} , there exists a function $g\in X'(\Omega)$ with $\|g\|_{X'(\Omega)}\leq 1$ such that 
    \begin{align*}
        \|A_\Q f\|_{X(\Omega)} &\leq 2\int_\Omega A_\Q f(x)\,g(x)\,dx\\
        &\leq 2 \sum_{Q\in \Q}\int_Q \fint_Q |f(y)|\,|g(x)|\,dx.
\intertext{Then by Hölder's inequality (Lemma \ref{Hölder}) applied twice, the Muckenhoupt condition, and property $\G(\Omega)$, we have that}
        & \leq 2\sum_{Q\in \Q} \int_Q |Q|^{-1} \big(\|\chi_Q\|_{X'(\Omega)}\|f\chi_Q\|_{X(\Omega)}\big)\,\chi_Q(x)g(x)\,dx \\
        &\leq 2\sum_{Q\in \Q}|Q|^{-1}\,\|\chi_Q\|_{X'(\Omega)}\,\|f\chi_Q\|_{X(\Omega)}\,\|\chi_Q\|_{X(\Omega)}\|g\, \chi_Q\|_{X'(\Omega)}\\
        & \leq 2[X(\Omega)]_{A(\Omega)}\sum_{Q\in \Q}\,\|f\chi_Q\|_{X(\Omega)}\,\|g\, \chi_Q\|_{X'(\Omega)}\\
        &\leq 2[X(\Omega)]_{A(\Omega)}[X(\Omega)]_{\G(\Omega)}\|f\|_{X(\Omega)}\,\|g\|_{X'(\Omega)}\\
        &\leq 2[X(\Omega)]_{A(\Omega)}[X(\Omega)]_{\G(\Omega)}\|f\|_{X(\Omega)}.
    \end{align*}
\end{proof}

\section{Sobolev spaces defined over $X(\Omega)$}
\label{section:sobolev}

In this section we define the Sobolev spaces $W_X^1(\Omega)$ and $H^1_X(\Omega)$, and prove that $W_X^1(\Omega)$ is a Banach space. Throughout this section we fix a domain $\Omega$ and a quasi-Banach function space $X(\Omega)$.  We begin by recalling some basic definitions from the classical theory of Sobolev spaces.  For more information, see Gilbarg and Trudinger~\cite{MR737190} or Ziemer~\cite{MR1014685}.  A function $f\in L^1_{loc}(\Omega)$ is weakly differentiable, or has derivatives in the distributional sense, with respect to $x_j$ if there exists a function $g_j \in L^1_{loc}(\Omega)$ such that for every $\phi\in C^\infty_c(\Omega)$,
\begin{equation*}
\int_\Omega g_j(x)\phi(x)\,dx = -\int_\Omega f(x)\partial_j\phi(x)\,dx.    
\end{equation*}
We denote this function $g_j$ by $\partial_j f$. Let $\nabla f=(\partial_1f,\dots,\partial_nf)$.  Define $W^{1,1}_{loc}(\Omega)$ to be the collection of functions $f\in L^1_{loc}(\Omega)$ such that $\partial_j f$ exists for $j=1,\dots,n$ and is in $L^1_{loc}(\Omega)$.  We will often write $\vecf=(f,\grad f)\in W^{1,1}_{loc}(\Omega)$.

\begin{defn}
    Given a quasi-Banach space $X(\Omega)$,  define $W^1_X(\Omega)$ to be the collection of all functions $\vecf=(f,\nabla f)\in W^{1,1}_{loc}(\Omega)$ such that
\begin{equation*}
    \|\vecf\|_{W^1_X(\Omega)}=\|f\|_{X(\Omega)}+\|\nabla f\|_{X(\Omega)} < \infty. 
\end{equation*} 
\end{defn}

\begin{prop} \label{prop:$W^1_X$ is Banach}
   Given a quasi-Banach space $X(\Omega)$,  suppose $\chi_Q \in X'(\Omega)$ for all cubes $Q\subset \Omega$.  Then $W^1_X(\Omega)$ is a quasi-Banach space.
\end{prop}

\begin{proof}
     Our proof is adapted from the proof of the second author and Penrod~\cite[Theorem~6.2]{MR4777231} and the second and third authors and Moen~\cite[Theorem~5.2]{MR3544941} in the setting of matrix weights.
    Since we can map each $\vecf \in W^1_X(\Omega)$ to a unique $(n+1)$-tuple $(f,\partial_1 f,\ldots,\partial_n f)$, we may consider $W^1_X(\Omega)$ as a linear subspace of $\bigoplus_{j=0}^{n}X(\Omega)$. Since $X(\Omega)$ is a quasi-Banach space, so is $\bigoplus_{j=0}^{n}X(\Omega)$. Therefore, to complete the proof it will suffice to show $W^1_X(\Omega)$ is a closed subspace of $\bigoplus_{j=0}^{n}X(\Omega)$. 
    
    Fix a Cauchy sequence $\{\vecf_k\}_{k=1}^\infty$ in $W^1_X(\Omega)$, $\vecf_k=(f_k,\partial_1 f_k,\ldots, \partial_n f_k)$. Since $X(\Omega)$ is complete, there exist $g, g_1,\dots,g_n \in X(\Omega)$ such that $f_k \rightarrow g$ and $\partial_jf_k \rightarrow g_j$ as $k\rightarrow\infty$ in $X(\Omega)$. For each $j$ we will show that $\partial_jg=g_j$, i.e., for all $\phi \in C^\infty_0(\Omega)$,
    \begin{equation}
        \int_\Omega g_j(x)\phi(x)\,dx=-\int_\Omega g(x)\partial_j\phi(x)\,dx. \label{weak der.}
    \end{equation}
    For in this case we have that $\vecf_k \to \vecg=(g,\partial_1 g,\ldots,\partial_n g) \in W^1_X(\Omega)$ and so $W^1_X(\Omega)$ is a closed subspace.
    
    Fix $\phi \in C^\infty_0(\Omega)$ and let $K=\supp(\phi)$. We first show both integrals are finite. For the first one, by Hölder's inequality (Lemma~\ref{Hölder}),
    \begin{equation*}
        \int_\Omega |g_j(x)\phi(x)|\,dx \leq \|\phi\|_{L^\infty(\Omega)}\int_\Omega|\chi_K \, g_j(x)|\,dx 
         \leq \|\phi\|_{L^\infty(\Omega)} \, \|\chi_K\|_{X^{'}(\Omega)} \, \|g_j\|_{X(\Omega)} < \infty.
    \end{equation*}
In the last inequality we use the fact that $\chi_Q\in X'(\Omega)$ for all cubes, which, by Lemma \ref{charac.func.of.comp.set.bdd.} and Remark~\ref{remark:X'-BFS-2}, implies that since $K\subset \Omega$ is compact, $\|\chi_{K}\|_{X{'}(\Omega)}<\infty$.
Essentially the same argument shows that the second integral in \eqref{weak der.} is finite since $\|\grad \phi\|_{L^\infty(\Omega)}<\infty$.

We now show that \eqref{weak der.} holds. Since $\vecf_k \in W^1_X(\Omega)$, for all $\phi \in C^\infty_0(\Omega)$,
\begin{equation*}
    \int_\Omega \partial_j f_k(x) \phi(x)\,dx=-\int_\Omega f_k(x)\partial_j\phi(x)\,dx. 
\end{equation*}
Therefore, for every $k$ we can estimate as follows:
\begin{align*}
    0 &\leq \bigg| \int_\Omega (g_j\phi+ g\partial_j\phi)\,dx \bigg| \\
    & = \bigg| \int_\Omega \big(g_j\phi - \partial_j f_k \phi - f_k\partial_j\phi+ g\,\partial_j\phi\big)\,dx \bigg| \\
    & \leq \bigg| \int_\Omega (g_j-\partial_j f_k) \phi \, dx \bigg| + \bigg| \int_\Omega (g - f_k) \partial_j\phi \, dx \bigg |\\
    &\leq \|\phi\|_{L^\infty(\Omega)}\, \|\chi_K\|_{X^{'}(\Omega)} \, \|g_j-\partial_j f_k\|_{X(\Omega)}\\
    & \quad \quad +\|\partial_j\phi\|_{L^\infty(\Omega)}\, \|\chi_K\|_{X^{'}(\Omega)}\, \|g-f_k\|_{X(\Omega)}.
\end{align*}
Both norms in the last line go to zero as $k\rightarrow\infty$, and so we must have that~\eqref{weak der.} holds.  This completes the proof.
\end{proof}

We now define $H^1_X(\Omega)$.  Let $\lip(\Omega)$ consist of all locally Lipschitz functions on $\Omega$.  Recall that by Rademacher's theorem, every function in $\lip(\Omega)$ has a locally bounded derivative; see~\cite[Section~3.1]{MR3409135}. 

\begin{defn}
    Given a quasi-Banach space $X(\Omega)$  define the space $H^1_X(\Omega)$ to be the formal closure of $\lip(\Omega)\cap W^1_X(\Omega)$  with respect to the  $W^1_X(\Omega)$ norm.
\end{defn}

\begin{remark}
    For an arbitrary QBFS $X(\Omega)$ we do not know if $\lip(\Omega)\cap W^1_X(\Omega)$ is non-empty.  Such a space seems quite pathological, and we will not concern ourselves with this possibility.
\end{remark}

By ``formal closure" we mean the collection of all equivalence classes of Cauchy sequences $\{\vecf_k\}_{k=1}^\infty$,  $\vecf_k=(f_k,\grad f_k)$.  Two sequences $\{\vecf_k\}_{k=1}^\infty$ and $\{\vecg_k\}_{k=1}^\infty$ are equivalent if $\lim_{k\rightarrow\infty}\|\vecf -\vecg\|_{W^1_X(\Omega)} = 0$.  Since $X(\Omega)$ is complete, the sequences $\{f_k\}_{k=1}^\infty$ and $\{\grad f_k\}_{k=1}^\infty$ converge in $X(\Omega)$ and so $\{\vecf_k\}_{k=1}^\infty$ converges to a pair $\vecf=(f,\vecg)\in X(\Omega)\bigoplus X(\Omega)$.  Since this pair is unique to the equivalence class containing $\{\vecf_k\}_{k=1}^\infty$, we can identify each equivalence class with such a pair.  However, {\em a priori}, we do not know that $\grad f=\vecg$ even in the sense of weak derivatives.  This phenomenon occurs in the study of  Sobolev spaces that are the solution spaces of degenerate elliptic equations.  See~\cite{tubitak1} for further information and references.  
But, with an additional assumption, we can avoid this problem and prove the following result.

\begin{prop}
     Given a quasi-Banach space $X(\Omega)$,  suppose $\chi_Q \in X'(\Omega)$ for all cubes $Q\subset \Omega$.  Then $H^1_X(\Omega)$ is non-empty and is contained in $W^1_X(\Omega)$.
\end{prop}

\begin{proof}
Note first that, given our hypothesis on $X'(\Omega)$, the argument to show that both integrals in~\eqref{weak der.} are finite also proves that $\lip_0(\Omega)$, the collection of all Lipschitz functions whose support is a compact subset of $\Omega$, is contained in $W^1_X(\Omega)$.  Using constant sequences, we also see that $\textrm{Lip}_0(\Omega)\subset H^1_X(\Omega)$.  Further, by Proposition~\ref{prop:$W^1_X$ is Banach}, $W_X^1(\Omega)$ is a quasi-Banach space, and so we immediately have that $\textrm{Lip}_0(\Omega)\subset H^1_X(\Omega)\subset W^1_X(\Omega)$. 
\end{proof}

\section{Density in quasi-Banach function spaces}
\label{section:density}

In this section we prove, with increasingly stricter hypotheses, that bounded functions of compact support, continuous functions of compact support, and smooth functions of compact support are  dense in a QBFS $X(\Omega)$.  To prove Theorem~\ref{thm:H=W} we only need the density of continuous functions.  We include the other two results because they are of independent interest and have relatively short proofs.

The proof of the density of bounded functions of compact support is adapted from that in the case of variable Lebesgue spaces in~\cite[Theorem~2.72]{MR3026953}.

\begin{lemma} \label{lemma:bounded-compact-dense}
    Let $X(\Omega)$ be a quasi-Banach function space with absolutely continuous norm.   Then the set of bounded functions with compact support, $L_c^{\infty}(\Omega)$, is dense in $X(\Omega)$. 
\end{lemma}

\begin{proof}
    Let $\{K_k\}_{k=1}^\infty$ be a nested sequence of compact subsets of $\Omega$ such that $\Omega=\bigcup_j{K_k}$.  For instance, let  $K_k=\{x \in \Omega : \dist(x,\Omega)\geq1/k\}\cap\overline{B(0,k)}$. Now fix any $f\in X(\Omega)$ and define the sequence $\{f_k\}_{k=1}^\infty$ by
    \begin{displaymath}
        {f_k(x)}=  \left\{\begin{array}{cc}
            k\,,&  f_k(x)>k, \\
            f(x)\,, & -k\leq f(x)\leq k, \\
            -k\,, & f_k(x)<-k.
        \end{array} \right.
    \end{displaymath}  
Let $g_k(x)=f_k(x)\,\chi_{K_k}(x)$. Then $g_k \in L_c^\infty(\Omega)$.  Further, $g_k\rightarrow f$ pointwise almost everywhere. Since $f\in X(\Omega)$ and $|g_k(x)|\leq |f(x)|$, by the ideal property we have that $g_k\in X(\Omega)$. Therefore, by Lemma~\ref{lemma:dominated-conv}, $g_k\to f$ in $X(\Omega)$.  
\end{proof}

We next show that continuous functions of compact support are dense.  The following result was proved by Tao, {\em et al.}~\cite[Prop.~3.8]{MR4568877} for ball quasi-Banach function spaces when $\Omega=\R^n$.  Their proof in turn is a modification of the classical proof that $C_c(\R^n)$ is dense in $L^1(\R^n)$.  To be clear on where our hypotheses are used, we give the short proof.

\begin{lemma}\label{Tao's Prop.}
    Let $X(\Omega)$ be a quasi-Banach function space with absolutely continuous norm and such that $\chi_Q\in X(\Omega)$ for all cubes $Q\subset \Omega$.  Then $C_c(\R^n)$ is dense in $X(\Omega)$.
\end{lemma}

\begin{proof}
    Fix $f\in X(\Omega)$ and $\epsilon>0$.  We will construct a function $u \in C_c(\Omega)$ such that $\|f-u\|_{X(\Omega)}<\epsilon$.  By a standard approximation argument, we may assume without loss of generality that $f$ is non-negative.  Then there exists a sequence $\{f_j\}_{j=1}^\infty$ of simple functions such that $f_j\leq f_{j+1}\leq f$ and $f_j\to f$ pointwise.   Therefore, by the ideal property and Lemma~\ref{lemma:dominated-conv}, there exists a simple function $g\in X(\Omega)$,
    \[ g = \sum_{k=1}^N \lambda_k \chi_{E_k}, \]
    such that $\|f-g\|_{X(\Omega)} < \frac{\epsilon}{3}$.  By the inner regularity of the Lebesgue measure, for each $k$ we can find a nested sequence of compact sets $E_k^j \subset E_k$ such that $\chi_{E_j^k}\to \chi_{E_k}$.  Therefore, again by the ideal property and Lemma~\ref{lemma:dominated-conv}, there exists a simple function $h\in X(\Omega)$,
    \[ h = \sum_{k=1}^N \lambda_k \chi_{F_k}, \]
    with the sets $F_k$ compact, such that $\|g-h\|_{X(\Omega)} < \frac{\epsilon}{3}$.  By the outer regularity of the Lebesgue measure, we can find a nested sequence of bounded, open sets $F_k^j$ containing $F_k$ such that $\chi_{F_k^j} \to \chi_{F_k}$.  Since $\chi_Q \in X(\Omega)$ for all cubes, and since the closure of each set $F_k^j$ is compact, by the ideal property and Lemma~\ref{charac.func.of.comp.set.bdd.}, $\chi_{F_k^j} \in X(\Omega)$.  Thus, by Lemma~\ref{lemma:dominated-conv}, there exists a simple function $\ell\in X(\Omega)$,
    \[ \ell = \sum_{k=1}^N \lambda_k \chi_{G_k}, \]
    with the sets $G_k$ open, such that $\|h-\ell\|_{X(\Omega)} < \frac{\epsilon}{3}$. Finally, by Urysohn's lemma (see Rudin~\cite[Lemma~2.12]{MR924157}), there exists $u\in C_c(\Omega)$ such that $h\leq u \leq \ell$.  Hence, by the ideal property, $\|h-u\|_{X(\Omega)}< \frac{\epsilon}{3}$.  Combining this with the previous inequalities we see that $u$ is the desired function.
\end{proof}

The density of smooth functions of compact support is a consequence of Theorem~\ref{thm.conv.op.bdd.} below, and so we defer the proof until the end of Section~\ref{section:convolution}.
A version of this result was proved by Dai~{\em et al.}~\cite[Proposition~3.8]{MR4700381} for ball quasi-Banach spaces when $\Omega=\R^n$, but assuming the uniform boundedness of a different family of averaging operators: the so-called centered ball averaging operators: for every $x\in\Omega$  and $r>0$,
\[ Bf(x) = \avgint_{B(x,r)} f(y)\,dy. \]
It is an open question to determine if Hypothesis~\ref{hyp:main} implies that these operators are uniformly bounded.

\begin{prop} \label{prop:smooth-dense}
    Let $X(\Omega)$ be a Banach space that satisfies Hypothesis~\ref{hyp:main}.  Then   $C^\infty_c(\Omega)$ is dense in $L^\infty_c(\Omega)$ in the $X(\Omega)$ norm.
\end{prop}

\section{Convolutions in quasi-Banach function spaces}
\label{section:convolution}

\begin{lemma}\label{lem.conv.CUBE.op.bdd.}
    Let  $X(\Omega)$ be a quasi-Banach function space that satisfies the strong Muckenhoupt condition.  If $f$ has compact support in $\Omega$, then there exists $\epsilon>0$, depending on $\supp(f)$,  such that given any cube $Q$ centered at the origin with $\ell(Q)<\epsilon$, $|Q|^{-1}\,\chi_Q*f \in X(\Omega)$ and
    \begin{equation*}
        \big\||Q|^{-1}\,\chi_Q*f\big\|_{X(\Omega)}\leq C(n)\|f\|_{X(\Omega)}.
    \end{equation*}
The same inequality is true if we replace the cube $Q$ with any ball $B=B(0,r)$ with $r<\epsilon/2$.
\end{lemma}

\begin{proof}
    Fix $f\in X(\Omega)$ with compact support and let $\delta = \dist({\supp(f)},\partial\Omega)>0$.  Fix $\epsilon=\delta/(3\sqrt{n})$ and let $Q$ be any cube centered at the origin with $\ell(Q)<\epsilon$.  For each $k\in\Z^n$, define $Q_k=Q+l(Q)k$; then $\{Q_k\}_{k\in\Z^n}$ forms a partition of $\R^n$. By our choice of $\epsilon$, $3Q_k\cap \supp(f) \neq \emptyset$ only if $3Q_k \subset \Omega$. Let $\mathcal{F} = \{ k \in \Z^n : 3Q_k\cap \supp(f) \neq \emptyset\}$.  
    
    Fix $k\in \Z^n$ and $x\in Q_k$. Then, since $Q$ is centered at the origin, $\{\,y\in \R^n : x-y\in Q\}\subseteq3Q_k$. 
    Hence, we have that
    \begin{align*}
        \big||Q|^{-1}(\,\chi_Q*f)(x)\big|
        & =\bigg||Q|^{-1}\int_{\R^n}f(y)\,\chi_Q(x-y)\,dy\bigg| \\
        & \leq |Q|^{-1}\int_{\R^n}|f(y)|\,\chi_Q(x-y)\,dy 
        & \leq C(n)\fint_{3Q_k} |f(y)|\,dy \chi_{3Q_k}(x). 
    \end{align*} 
    Thus, $|Q|^{-1}(\chi_Q*f)(x)\neq 0$ only if $k \in \mathcal{F}$ and so
\begin{equation*}
    \big\||Q|^{-1}\,\chi_Q\,*f \big\|_{X(\Omega)}
    =\bigg\| \sum_{k\in \Z^n}(|Q|^{-1}\chi_Q*f)\chi_{Q_k} \bigg\|_{X(\Omega)}
      \leq C(n)\bigg\| \sum_{k\in \mathcal{F}}
     \avgint_{3Q_k} |f(y)|\,dy \chi_{3Q_k} \bigg\|_{X(\Omega)}.
\end{equation*}
Since the cubes $Q_k$ are pairwise disjoint, we can divide the cubes $\{3Q_k\}_{k\in \mathcal{F}}$ into  $3^n$ families $\Q_i$ of pairwise disjoint cubes. Therefore, by the strong Muckenhoupt condition,
\begin{multline*}
C(n) \bigg\| \sum_{k\in \mathcal{F}}
     \avgint_{3Q_k} |f(y)|\,dy \chi_{3Q_k} \bigg\|_{X(\Omega)}
    \leq C(n)\sum_{i=1}^{3^n} \|A_{\Q_i}(|f|)\|_{X(\Omega)} \\
    \leq C(n) [X(\Omega)]_{A^s(\Omega)} \sum_{i=1}^{3^n} \|f\|_{X(\Omega)}
    = C(n) [X(\Omega)]_{A^s(\Omega)} \|f\|_{X(\Omega)}.
\end{multline*}
This complete the proof for cubes.

\medskip

To prove this result for the balls, fix a ball $B=B(0,r)$, $r<\epsilon/2$, and let $Q$ be the  smallest cube centered at the origin containing $B$. Then $\ell(Q)<\epsilon$ and $|B|\approx|Q|$, and so for any $x\in \Omega$,
\begin{multline*}
    \big||B|^{-1}(\,\chi_B*f)(x)\big|
    =\bigg||B|^{-1}\int_{\R^n}f(y)\,\chi_B(x-y)\,dy\bigg| \\
        \leq |B|^{-1}\int_{\R^n}|f(y)|\,\chi_B(x-y)\,dy 
        \leq C(n)|Q|^{-1}\int_{\R^n} |f(y)|\chi_Q(x-y)\,dy = |Q|^{-1}(\chi_Q*f)(x).
\end{multline*}
The desired inequality for balls now follows from the one for cubes.  This completes the proof.
\end{proof}

\medskip

In the proof of the following result, we require that $X(\Omega)$ is a Banach function space so that we can apply  the triangle inequality (as opposed to the quasi-triangle inequality with $M_X>1$). 

\begin{theorem} \label{thm.conv.op.bdd.}
 Let  $X(\Omega)$ be a Banach function space that satisfies Hypothesis~\ref{hyp:main}.
    Let $\phi\in C^\infty_c(B(0,1))$ be a nonnegative, radially symmetric and decreasing function with $\int_{\R^n}\phi(x)\,dx=1$. For $t>0$, let $\phi_t(x)=t^{-n}\phi(x/t)$. If $f$ has compact support in $\Omega$, then there exists $\eta>0$, depending on $\supp(f)$, such that 
    \begin{equation}
        \sup_{0<t<\eta}\|\phi_t*f\|_{X(\Omega)}\leq C(n)\|f\|_{X(\Omega)} \label{ineq.conv.op.bdd.}
    \end{equation}
     Moreover, we have that for all such $f\in X(\Omega)$,
     \begin{equation}
         \lim_{t\rightarrow 0}\|\phi_t*f -f\|_{X(\Omega)}=0. \label{conv.op.converges}
     \end{equation}
\end{theorem}

\begin{proof}
   We first prove inequality \eqref{ineq.conv.op.bdd.}.  
   Fix $f\in X(\Omega)$ with compact support.  Define the function 
    \begin{equation*}
        \Phi(x)=\sum_{k=1}^\infty a_k\,|B_k|^{-1}\, \chi_{B_k}(x).
    \end{equation*}
where $\{B_k\}_{k=1}^\infty$ is a sequence of balls, each centered at the origin with $B_{k+1}\subset B_k $ for all $k$, and with $r(B_k) < \epsilon/2$, where $\epsilon>0$ is as in Lemma~\ref{lem.conv.CUBE.op.bdd.}. We also assume that  $a_k\geq 0$ with $\sum_{k=1}^\infty a_k \leq 1$.  We will first prove that
\begin{equation}
    \|\Phi*f\|_{X(\Omega)} \leq C(n) \|f\|_{X(\Omega)}. \label{conv.op.bdd.in-proof}
\end{equation}
Since $X(\Omega)$ is a Banach function space with the Fatou property, by the triangle inquality and  Lemma \ref{lem.conv.CUBE.op.bdd.},
\begin{align*}
    \|\Phi*f\|_{X(\Omega)}
    &= \bigg\|\sum_{k=1}^\infty a_k \, |B_k|^{-1}\,(\chi_{B_k}\,*f) \bigg\|_{X(\Omega)} \\ 
    &\leq\sum_{k=1}^\infty a_k  \big \| |B_k|^{-1}\,\chi_{B_k}\,*f\, \big\|_{X(\Omega)} \\
    &\leq C(n) \sum_{k=1}^\infty a_k\,\|f\|_{X(\Omega)}\\
    &\leq C(n)\|f\|_{X(\Omega)}.
\end{align*}

We can now prove~\eqref{ineq.conv.op.bdd.}.  Let $\eta= \epsilon/2$, with $\epsilon$ as above.  Fix $\phi$ as in the hypotheses and fix $t$, $0<t<\eta$.
Since $|\phi_t*f(x)| \leq (\phi_t*|f|)(x)$, we may assume without loss of generality that $f$ is non-negative.  
Approximate $\phi_t$ from below by an increasing sequence of functions  $\Phi_l(x)=\sum_{k=1}^\infty a_k^l\,|B_k|^{-1}\, \chi_{B_k}(x)$ with the same properties as $\Phi$ above. Then by the monotone convergence theorem, $\Phi_l*f(x)$ increases pointwise to $\phi_t*f(x)$.    Therefore, by the Fatou property and inequality~\eqref{conv.op.bdd.in-proof},
\begin{equation*}
    \|\phi_t*f\|_{X(\Omega)} = \lim_{l\rightarrow\infty}\|\Phi_l*f\|_{X(\Omega)}\leq C(n)\|f\|_{X(\Omega)}.
\end{equation*}

\medskip

We now prove  \eqref{conv.op.converges}. Fix $\epsilon>0$. By Lemma~\ref{Tao's Prop.}, $C_c(\Omega)$ is dense in $X(\Omega)$, and so there exists $g\in C_c(\Omega)$ such that $\|f-g\|_{X(\Omega)}<\epsilon$. By a well-known property of convolutions (see~\cite[Theorem~8.14]{MR1681462}), since $g$ is uniformly continuous $\phi_t*g \to g$ uniformly.  Since $g$ has compact support, there exists $t_0>0$ such that for all $0<t \leq t_0$ , $\phi_t*g$ has compact support and 
$\supp(g)\subset \supp(\phi_t*g) \subset \supp(\phi_{t_0}*g) \subset \Omega$. Let $K = \supp(\phi_{t_0}*g)$.  Then we have that 
\[ \|\phi_t*g-g\|_{X(\Omega)} \leq \|\phi_t*g-g\|_{L^\infty(\Omega)}\|\chi_{K}\|_{X(\Omega)}. \]
By Lemma~\ref{charac.func.of.comp.set.bdd.}, $\|\chi_{K}\|_{X(\Omega)}<\infty$, so the limit of the right-hand side goes to zero as $t\to 0$ is $0$. Therefore, since $f-g$ has compact support, by~\eqref{ineq.conv.op.bdd.},
\begin{align*}
    & \limsup_{t\to0} \|\phi_t*f-f\|_{X(\Omega)} \\
    & \qquad \qquad \leq \limsup_{t\to 0}\big(\|\phi_t*f-\phi_t*g\|_{X(\Omega)}+\|\phi_t*g-g\|_{X(\Omega)}+\|g-f\|_{X(\Omega)}\big) \\
    &\qquad \qquad \leq (1+C(n))\|f-g\|_{X(\Omega)}+ \lim_{t\to0}\|\phi_t*g-g\|_{X(\Omega)}\\
    & \qquad \qquad \leq (1+C(n))\epsilon.
\end{align*}
Since $\epsilon>0$ was arbitrary, \eqref{conv.op.converges} follows at once.
\end{proof}

We conclude this section with the proof of Proposition~\ref{prop:smooth-dense}.

\begin{proof}
    Fix $\epsilon>0$.  Then by Lemma~\ref{lemma:bounded-compact-dense}, there exists $g\in L_c^\infty(\Omega)$ such that $\|f-g\|_{X(\Omega)} < \frac{\epsilon}{2}$.  Since $g$ has compact support, by Theorem~\ref{thm.conv.op.bdd.}, for $\phi$ as in the statement of the theorem we can find $t>0$ such that $\|\phi_t*g-g\|_{X(\Omega)} < \frac{\epsilon}{2}$.  By the properties of convolution operators (see~\cite[Prop.~8.10]{MR1681462}), $\phi_t*g \in C_c^\infty(\Omega)$.  The desired density follows immediately.
\end{proof}

\section{Proof of Theorem \ref{thm:H=W}} 
\label{section:main-proof}

In this section we prove our main result, Theorem~\ref{thm:H=W}, and Corollary~\ref{cor:rn-compact-support}.  Their proofs adapt arguments from the proofs of the corresponding matrix-weighted results in~\cite{MR4777231,MR3544941}.  These in turn were based on the original proof in~\cite{MR164252}.

\begin{proof}[Proof of Theorem~\ref{thm:H=W}]
By Proposition~\ref{prop:$W^1_X$ is Banach} we have that $H_X^1(\Omega)\subset W_X^1(\Omega)$.  Therefore, it remains to prove the reverse inclusion.  Fix  $f \in W^1_X(\Omega)$ and any $\epsilon>0$; we will show that there exists $g \in C^\infty(\Omega, \R^n)\cap W^1_X(\Omega)$ such that $\|f-g\|_{W^1_X(\Omega)} < \epsilon$, which implies the desired conclusion.

For each $k\in \N$, define the bounded sets 
    \begin{equation*}
        \Omega_k=\{x\in \Omega : |x|<k, \, \dist(x,\partial\Omega)>1/k\}.
    \end{equation*}
Let $\Omega_0=\Omega_{-1}=\emptyset$, and  define the sets $A_k=\Omega_{k+1} \setminus\overline{\Omega}_{k-1}$. The sets $\{A_k\}_{k=0}^\infty$ form an open cover of $\Omega$, each $\overline{A}_k$ is compact, and the cover is locally finite:  given a compact set  $K \subset \Omega$,  $K\cap A_k\neq \emptyset$ for only a finite number of indices $k$. Thus, we can form a partition of unity subordinate to this cover: there exists $\psi_k\in C^\infty_c(A_k)$ such that for all $x \in \Omega$, $0\leq \psi_k(x)\leq 1$ and $\sum_{k=0}^\infty \psi_k(x)=1$.  (See Giusti~\cite[Theorem~3.1]{MR1962933}.)

For each $k$ we have that $|\psi_kf| \leq |f|$, and using the ideal property we see that $\psi_kf \in X(\Omega)$.  Since 
$\phi\in C^\infty_c(\Omega)$ and $f\in W^{1,1}_{loc}(\Omega)$, $\phi f \in W^{1,1}_{loc}(\Omega)$.  Moreover,  by the product rule (see~\cite[eqn.~(7.18)]{MR737190},  $|\grad(\psi_kf)| \leq |\psi_k \grad f| + |f\grad \psi_k| \leq |\grad f|+\|\psi_k\|_{L^\infty(\Omega)}|f|$. Thus, again by the ideal property, we have $\grad(\psi_kf)\in X(\Omega)$ and so $\psi_kf \in W_X^1(\Omega)$.  

Fix a nonnegative radially symmetric and decreasing function $\phi\in C^\infty_c(B(0,1))$ with $\int_{B(0,1)}\phi(x)\,dx=1$. For all $t>0$, define $\phi_t(x)=t^{-n}\phi(x/t)$. Then the convolution
\begin{equation*}
    \phi_t*(\psi f)(x)=\int_{A_k}\phi_t(x-y)\psi_k(y)f(y)dt 
\end{equation*}
is only non-zero if for some $y\in A_k, |x-y|<t$. Hence, for $k\geq3$, if we take $t=t_k$ such that $0<t_k<(k+1)^{-1}-(k+2)^{-1}$, then, there exists $y\in A_k$ with $|x-y|<t_k$ only if 
\begin{equation*}
    (k+2)^{-1}<\dist(x,\partial\Omega)\leq(k-2)^{-1}.
\end{equation*}
So, $\phi_{t_k}*(\psi_kf)(x)$ is non-zero only if $ x\in\Omega_{k+2}\setminus\overline{\Omega}_{k-2} = B_k$.  
We will fix the precise value of each $t_k$ below.

Define
\begin{equation*}
    g(x)=\sum_{k=1}^{\infty}\phi_{t_k}*(\psi_kf)(x).
\end{equation*}
Since $\phi\in C^\infty_c(\Omega)$, by the properties of convolutions, each summand is in $C^\infty_c(\Omega)$. (See~\cite[Prop.~8.10]{MR1681462}.) Moreover, for each $x\in\Omega$, $x$ is in a finite number of the sets  $B_k$.  Thus, the series converges locally uniformly and $g\in C^\infty(\Omega)$.

Finally, we claim we can choose $t_k$ so that $\|\vecf-\vecg\|_{W^1_X(\Omega)}<\epsilon$. To prove this, we consider each part of the norm  separately. Since $\psi_k f$ has compact support, by Theorem \ref{thm.conv.op.bdd.},  for each $k$, we may choose $s_k>0$ sufficiently small  that $\|\psi_k f-\phi_{s_k}*(\psi_k f)\|_{X(\Omega)}<\epsilon/2^{k+1}$.
But then, by the triangle inequality and the Fatou property,
\begin{align*}
    \|f-g\|_{X(\Omega)}&=\Bigg\|\sum_{k=1}^\infty(\psi_k f-\phi_{s_k}*(\psi_k f)) \Bigg\|_{X(\Omega)}\\
    &\quad\quad\quad\quad\leq \sum_{k
    =1}^\infty \|(\psi_k f-\phi_{s_k}*(\psi_k f))\|_{X(\Omega)}\leq \sum_{k=1}^\infty\frac{\epsilon}{2^{k+1}}=\frac{\epsilon}{2}. 
\end{align*}

The argument for the norm of the gradient is similar. By Theorem~\ref{thm.conv.op.bdd.}, for each $k$, we may choose $t_k\leq s_k$ sufficiently small so that for each $j=1,\dots,n$,
\begin{equation*}
    \|\partial_j(\psi_kf)-\phi_{t_k}*\partial_j(\psi_kf)\|_{X(\Omega)}<\frac{\epsilon}{n2^{k+1}}. 
\end{equation*}
Then, again by the triangle inequality and the Fatou property,  for each $j$,
\begin{multline*}
    \|\partial_j(f-g)\|_{X(\Omega)}
    =\bigg\|\sum_{k=1}^\infty \big(\partial_j(\psi_kf)-\phi_{t_k}*\partial_j(\psi_kf)\big)\bigg\|_{X(\Omega)}\\
   \leq \sum_{k=1}^\infty\big\|\partial_j(\psi_kf)-\phi_{t_k}*\partial_j(\psi_kf)\big\|_{X(\Omega)} < \frac{\epsilon}{2n}.
\end{multline*} 
If we combine these estimates we see that 
\begin{equation*}
    \|f-g\|_{W^1_X(\Omega)}=\|f-g\|_{X(\Omega)}+\sum_{j=1}^n\big\|\partial_j(f-g)\big\|_{X(\Omega)}<\frac{\epsilon}{2}+\sum_{j=1}^n\frac{\epsilon}{2n}=\epsilon.
\end{equation*}
Since $\epsilon>0$ was arbitrary, we see that $f \in H^1_X(\Omega)$.  This completes the proof.
\end{proof}

\medskip

\begin{proof}[Proof of Corollary~\ref{cor:rn-compact-support}]
    Given $f\in W^1_X(\R^n)$ and $\epsilon>0$ we need to find $g\in C^\infty_c(\R^n)$ so that $\|f-g\|_{W^1_X(\Omega)}<\epsilon$.  To do this, we begin by using Theorem \ref{thm:H=W} with $\Omega=\R^n$ to choose $h \in C^\infty(\R^n)\cap W^1_X(\R^n)$ so that $\|f-h\|_{W^1_X(\Omega)}<\frac{\epsilon}{2}$.  %when $f\in W^1_X(\Omega)$ and $g\in C_c^\infty(\R^n)$. So, we need to construct a function $g$ and show
    %\begin{equation*}
     %\|f-g_k\|_{W^1_X(\Omega)} \leq \|f-h\|_{W^1_X(\Omega)}+\|h-g_k\|_{W^1_X(\Omega)}<\epsilon  . 
    %\end{equation*}
%The first norm in the right-hand side holds since we know $ W=H=\overline{C^\infty \cap W^1_X(\Omega)}$ by Theorem ~\ref{thm:H=W}. So we have an $h\in \overline{C^\infty \cap W^1_X(\Omega)}$ such that
%\begin{equation*}
 %   \|f-h\|_{W^1_X(\Omega)}<\epsilon/2 .
%\end{equation*}
To finish the proof, we construct a sequence $\{g_k\}_k\subset C^\infty_c(\R^n)$ that converges to $h$ in $W^1_X(\R^n)$.  %But we need the function $h$ to be in $C_c^\infty(\Omega)$. 
To do this we use a sequence of cut-off functions. For each $k\geq 2$, let $\eta_k\in C_c^\infty(\R^n)$ be such that $\supp(\eta_k)\subset B(0,2k)$, $0 \leq \eta_k\leq 1$, $\eta_k(x)=1$ for each $x\in B(0,k)$ and $|\nabla \eta_k|\lesssim 1/k$ on the support of $\nabla \eta_k$. For each $k$ set $g_k=h\eta_k \in \C_c^\infty(\R^n)$ and notice that $|g_k|=|h \eta_k| \leq|h|$ for every $x$. Since $\eta_k\rightarrow 1$, $g_k \rightarrow h$ pointwise as $k \rightarrow \infty$. By the dominated convergence theorem, Lemma~\ref{lemma:dominated-conv}, we find
\begin{equation*}
 \lim_{k\rightarrow \infty}\|g_k-h\|_{X(\Omega)}=0  . 
\end{equation*}
Similarly, $\nabla g_k=\nabla (h\eta_k)=\eta_k\nabla h +h\nabla \eta_k $. When $k\rightarrow \infty$, $\eta_k \rightarrow 1$ so $\nabla \eta_k\rightarrow 0$. Hence, $\nabla g_k \rightarrow \nabla h$ pointwise as $k\rightarrow \infty$. We also have $|\nabla g_k|\lesssim |\nabla h|+|h|$ and, again by dominated convergence, we get
\begin{equation*}
    \lim_{k\rightarrow \infty}\|\nabla g_k-\nabla h\|_{X(\Omega)}=0 .
\end{equation*}
If we combine these estimates,  we see that there exists  $K\in \N$ so that $\|h-g_K\|_{W^1_X(\Omega)}\leq \epsilon/2$. Hence, by  the triangle inequality, 
\[\|f-g_K\|_{W^1_X(\R^n)} \leq \|f-h\|_{W^1_X(\R^n)} + \|h-g_K\|_{W^1_X(\R^n)} < \epsilon.\]
\end{proof}

\bibliographystyle{plain}
\bibliography{HX=WX}

\end{document}